\documentclass[12pt]{article}
\usepackage[utf8]{inputenc}
\usepackage[T1]{fontenc}

\usepackage{hyperref}

\IfFileExists{definice}{

\usepackage{graphicx}
\usepackage{amsmath,amsthm,amssymb}
\usepackage{enumerate}
\usepackage{cleveref} 
\usepackage{subfig}
\usepackage{color}

\addtolength{\voffset}{-1cm} 
\addtolength{\hoffset}{-1.5cm} 
\setlength{\textheight}{22cm} \setlength{\textwidth}{16cm}

\def\Z{\mathbb Z}

\def\N{\mathbb N}
\def\A{\mathcal A}

\def\C{\mathcal C}

\def\P{\mathcal P}

\def\P{\mathcal P}
\def\PT{{\mathcal P}_{\Theta}}

\def\L{\mathcal L}
\def\Lu{{\mathcal L}(\uu)}
\def\uu{\mathbf u}

\def\tt{\mathbf t}

\def \id {{\rm Id}}
\def\PalT{{\rm Pal}_{\Theta}}

\def \Rk#1 {$\mathcal{R}_{#1}$}
\def \Rext {{\rm Rext}}
\def \Lext {{\rm Lext}}
\def \Bext {{\rm Bext}}
\def \Pext {{\rm Pext}}
\def \PextT {\Pext_{\Theta}}
\def \b {{\rm b}}

\def \FC#1 {
\mathcal{C}
\ifthenelse{\equal{#1}{}}{}{(#1)}
}

\def \PC#1 {
\mathcal{P}_{\Theta}
\ifthenelse{\equal{#1}{}}{}{(#1)}
}

\def \PCn#1 {
\mathcal{P}
\ifthenelse{\equal{#1}{}}{}{(#1)}
}

\def \Tr {R}
\def \Ta {\Theta_1}
\def \Tb {\Theta_2}

\newtheorem{thm}{Theorem}

\newtheorem{coro}[thm]{Corollary}

\newtheorem{lem}[thm]{Lemma}

\newtheorem{prop}[thm]{Proposition}

\crefname{thm}{theorem}{theorems}
\crefname{thrm}{theorem}{theorems}
\crefname{coro}{corollary}{corollaries}
\crefname{example}{example}{examples}
\crefname{lem}{lemma}{lemmas}
\crefname{lmm}{lemma}{lemmas}
\crefname{claim}{claim}{claims}
\crefname{obs}{observation}{observations}
\crefname{proposition}{proposition}{propositions}
\crefname{prop}{proposition}{propositions}
\crefname{defi}{definition}{definitions}

\theoremstyle{remark}

\crefname{example}{example}{examples}
}{}

\begin{document}


\title{Generalized Thue-Morse words and \\ palindromic richness}

\author{\v St\v ep\'an Starosta \\
  \multicolumn{1}{p{.7\textwidth}}{\centering\emph{ \small Department of Mathematics, FNSPE, Czech Technical University in Prague, Trojanova 13, 120~00 Praha~2, Czech Republic}}}

\date{}

\maketitle

\begin{abstract}
We prove that the generalized Thue-Morse word $\mathbf{t}_{b,m}$ defined for $b \geq 2$ and $m \geq 1$ as
$\mathbf{t}_{b,m} = \left ( s_b(n) \mod m \right )_{n=0}^{+\infty}$, where
$s_b(n)$ denotes the sum of digits in the base-$b$ representation of the integer $n$,
has its language closed under all elements of a group $D_m$ isomorphic to the dihedral group of order $2m$ consisting of morphisms and antimorphisms.
Considering simultaneously antimorphisms $\Theta \in D_m$, we show that $\mathbf{t}_{b,m}$ is saturated by $\Theta$-palindromes up to the highest possible level.
Using the terminology generalizing the notion of palindromic richness for
more antimorphisms recently introduced by the author and E. Pelantov\'a, we show that $\mathbf{t}_{b,m}$ is $D_m$-rich.
We also calculate the factor complexity of $\mathbf{t}_{b,m}$.
\end{abstract}



\section{Introduction}

Palindrome is a word which coincides with its reversal image, or
more formally, $w=\Tr(w)$, where $w$ is a finite word and the reversal (or mirror) mapping  $\Tr: \A^* \mapsto \A^*$ is defined by $\Tr(w_1w_2\ldots w_n)
=  w_nw_{n-1}\ldots w_1$ for letters $w_i \in \A$, $\A$ being an alphabet.
In \cite{DrJuPi}, the authors gave an upper bound on the number of palindromic factors in a finite word:
a finite word of length $n$ contains at most $n+1$ palindromic factors.
If this bound is attained, we say that a word is rich in palindromes (see \cite{GlJuWiZa}, or \cite{BrHaNiRe}, where such a word is called full).
This definition can be naturally extended to infinite words: an infinite word is rich in palindromes if every its factor is rich in palindromes.
For infinite words with language closed under reversal, i.e., containing with every factor also its reversal image,
 there exist several equivalent characterizations of palindromic richness.
Each of them can be adopted as definition.

Let us list three of these characterizations.
An infinite word $\uu$ with language closed under reversal is rich if one of the following statements holds:
\begin{enumerate}
 \item \label{rich1} any factor $w$ of $\uu$ of length $n$ contains exactly $n+1$ palindromic factors;
 \item \label{rich2} for any $n \in \N$, the equality $\Delta \C (n) + 2  =  \P(n) +  \P(n+1)$ is satisfied, where  $\Delta \C (n)$ denotes the first difference of the factor complexity of $\uu$ and $\P(n)$ denotes the palindromic complexity of $\uu$ (\cite{BuLuGlZa}).
\item \label{rich3} each complete return word of any palindrome occurring in $\uu$ is a palindrome as well (\cite{GlJuWiZa});
\end{enumerate}

Let us mention that the inequality
\begin{equation} \label{balazi} \Delta \C (n) + 2 \geq  \P(n) +
 \P(n+1) \end{equation}
  is valid for any infinite  word $\uu$  with language
 closed under reversal and for any $n$ (see \cite{BaMape}).
 Thus, both characterizations \ref{rich1} and \ref{rich2} express that a word $\uu$ rich in palindromes is saturated
 by palindromes up to the highest possible level.
Arnoux-Rauzy words (among them Sturmian words) and words
coding interval exchange transformation with the symmetric permutation belong to
the most prominent examples of rich words.

If we replace the reversal mapping
$\Tr$  by an antimorphism $\Theta$, we can define
\mbox{$\Theta$-palindromes} as words which are fixed points of $\Theta$,
i.e., $w=\Theta(w)$.
For any antimorphism $\Theta$, the notion of
$\Theta$-palindromic richness can be introduced and an analogy of
characterizations 1, 2 and 3 mentioned above can be formulated,
see \cite{Sta2010}.

Although the language of the Thue-Morse word is closed under two antimorphisms, it is not $\Theta$-rich for any of these two antimorphisms.
The author together with E. Pelantov\'a in \cite{PeSta1} explored
 infinite words with language closed under more
antimorphisms simultaneously.
For a given finite group $G$ formed by morphisms and antimorphisms
on $\A^*$, the words with language closed under
any element of $G$ are investigated.
If, moreover, such a word $\uu$ is uniformly recurrent, then a generalized version of the inequality \eqref{balazi} is proved: there exists an integer $N$ such
that
\begin{equation}\label{nerovnost}
\Delta \C (n) + \# G \ \ \geq   \sum_{\Theta \in
G^{(2)}}\Bigl(\P_{\Theta}(n) + \P_{\Theta}(n+1)\Bigr) \qquad
\hbox{for any} \ n \geq N\,, \end{equation}
 where $G^{(2)}$
denotes the set of involutive antimorphisms of $G$ and
$\P_{\Theta}$ is the $\Theta$-palindromic complexity, i.e.,
$\P_{\Theta}(n)$ counts the number of $\Theta$-palindromes of
length $n$ in the language of $\uu$.

Infinite uniformly recurrent words with language closed under all elements of $G$ and for which the equality  in \eqref{nerovnost}  is
attained for any $n$ greater than some integer $M$ are called \textit{almost $G$-rich}. Let us
emphasize that if the group $G$ contains besides
the identity only the reversal mapping $\Tr$,  the notion of almost
richness (as introduced in \cite{GlJuWiZa}) and the notion of almost
$G$-richness coincide.
In \cite{PeSta1}, \textit{$G$-richness} is also introduced.
Again, in case of $G = \{ \Tr, \id \}$ it coincides with classical palindromic richness.
The definition requires further notions, thus we omit it here and restrict ourselves to the following criterion of $G$-richness.
\begin{prop} \label{G_rich_criterion}
An infinite word with language closed under all elements of $G$ is $G$-rich if
\begin{itemize}
\item for any two antimorphisms $\Theta_1, \Theta_2 \in G$ and any non-empty $v \in \Lu$ it holds
$\Theta_1\neq \Theta_2 \Rightarrow  \Theta_1(v) \neq \Theta_2(v)$, and
\item for any two morphisms $\varphi_1, \varphi_2 \in G$ and any non-empty $v \in \Lu$  it holds \linebreak
$\varphi_1\neq \varphi_2 \Rightarrow  \varphi_1(v) \neq
\varphi_2 (v)$, and
\item the equality \eqref{nerovnost} is attained for all $n \geq 1$.
\end{itemize}
\end{prop}

In \cite{PeSta1}, the authors show that the Thue-Morse word is $G$-rich, where $G$ is a group generated by the reversal mapping and the antimorphism exchanging $0$ and $1$.
The class of so-called generalized Thue-Morse words is also partially treated in the article, but the question of their $G$-richness is not resolved.
In this article we give a proof that all generalized Thue-Morse words are $G$-rich and we give explicitly the group $G$.

The generalized Thue-Morse sequences were already considered by E. Prouhet in 1851, see \cite{Prouhet}.
Let $s_b(n)$ denote the sum of digits in the base-$b$ representation of the integer $n$, for integers $b \geq 2$ and $m \geq 1$.
The generalized Thue-Morse word $\tt_{b,m}$ is defined as
$$
\tt_{b,m} = \left ( s_b(n) \mod m \right )_{n=0}^{+\infty}.
$$
Using this notation, the famous Thue-Morse word equals
$\tt_{2,2}$. The word $\tt_{b,m}$ is over the alphabet
$\{0,1,\ldots, m-1\} = \Z_m$.
Similarly to the classical Thue-Morse
word, also $\tt_{b,m}$  is a fixed point of a  primitive
substitution, as already mentioned in \cite{AlSh}.
It is easy to see  that the substitution fixing the word $\tt_{b,m}$ is defined by
$$ \varphi_{b,m}(k) =  k(k + 1)(k + 2) \ldots \bigl(k + (b-1) \bigr) \quad \text{for any} \ k \in \Z_m,
$$
where the letters are expressed modulo $m$.
As already stated in \cite{AlSh}, it can be shown that $\tt_{b,m}$ is periodic if and only if $b \equiv 1 \mod m$.

We show that for any
parameters $b$ and $m$ the language of the word $\tt_{b,m}$ is
closed under all elements of a group, denoted $D_m$, isomorphic to the dihedral group of order $2m$, see \Cref{TM_closed_under_Dm},
 and that $\tt_{b,m}$  is $D_m$-rich, see \Cref{TM_je_Gr}.
In the last section, we use the results to give the formula for the factor complexity of $\tt_{b,m}$.

\section{Preliminaries}
An \textit{alphabet} $\A$ is a finite set, its elements are called \textit{letters}.
A \textit{finite word} over $\A$ is a finite string $w = w_1w_2\ldots w_n$
of letters $w_i\in \A$.
Its length is $|w| = n$.
The set $\A^*$ is formed by all finite words and it is a free monoid with the empty word $\varepsilon$ as neutral element.
An \textit{infinite word} $\uu = (u_i)_{i=0}^{+\infty}$ is a sequence of letters $u_i \in \A$.
A word $v \in \A^*$ is a \textit{factor} of a word $w$ (finite or infinite) if there exist words $s,t\in \A^*$ such that $w=svt$.
If $s=\varepsilon$, then $v$ is a prefix of $w$, if  $t=\varepsilon$, then $v$ is a suffix of $w$.
An integer $i$ such that $v = w_i \ldots w_{i+|w|-1}$ is called an \text{occurrence} of $v$ in $w$.
We say that an infinite word is  \textit{uniformly recurrent} if for each factor the gaps between its successive occurrences are bounded.

 By $\L_n(\uu)$ we denote the set of factors of $\uu$ of length $n$.
 The set of all factors of $\uu$ is denoted by $\Lu$ and is called the language of $\uu$.
 The \textit{factor complexity} $\C$ of an infinite word $\uu$ is a mapping $\N \mapsto \N$ counting the number
 of distinct factors of a given length, i.e., $\C(n) = \# \L_n(\uu)$.

A letter $a \in \A$ is a \textit{left extension} of a factor $w \in \Lu$
if $aw \in \Lu$.
If a factor has at least two distinct left extensions, then it is said to be \textit{left special}.
The set of all left extensions of $w$ is denoted $\Lext(w)$.
The definition of right extensions, $\Rext(w)$ and right special is analogous.
A factor which is left and right special is called \textit{bispecial} (BS).

The \textit{bilateral order} $\b(w)$ of a factor $w$ is the number
$\b(w) :=  \# \Bext(w) -  \# \Lext(w) - \# \Rext(w) + 1$
where the quantity $ \Bext(w) = \{ awb \in \Lu \mid a,b \in \A \}$.
In \cite{Ca}, the following relation between the second difference of factor complexity and bilateral orders is shown:
\begin{equation} \label{C_je_suma_bw}
\C(n+2) - 2 \C(n+1) + \C(n) = \Delta^2 \C(n) = \sum_{w \in \L_n(\uu)} \b(w).
\end{equation}
One can easily show that if $w$ is not a bispecial factor, then $\b(w) = 0$.
Thus, to enumerate the factor complexity of an infinite word $\uu$
one needs to calculate $\C(0), \C(1)$, and bilateral orders of all its bispecial factors.

A  mapping $\varphi$ on $\A^*$ is called a \textit{morphism} if $ \varphi(vw)=  \varphi(v) \varphi(w)$ for any $v,w
 \in \mathcal{A}^*$;
 an \textit{antimorphism} if $ \varphi(vw)=  \varphi(w) \varphi(v)$ for
any $v,w \in \mathcal{A}^*$.
By $AM(\A^*)$ we denote the set of all morphisms and antimorphisms over $\A^*$.
Let $\nu \in AM(\A^*)$.
We say that $\Lu$ is \textit{closed under $\nu$}
if for all $w \in \Lu$ we have $\nu(w) \in \Lu$.

It is clear that the reversal mapping $\Tr$ is an antimorphism.
Moreover, it is an involution, i.e., $\Tr^2 = \id$.
A fixed point of an antimorphism $\Theta$  is called \textit{$\Theta$-palindrome}. 
If $\Theta = \Tr$, then we say palindrome or classical palindrome instead of $\Tr$-palindrome.
The set of all $\Theta$-palindromic factors of an infinite word $\uu$ is denoted by $\PalT(\uu)$.
The \textit{$\Theta$-palindromic complexity} of $\uu$ is the mapping $\PT: \N \mapsto \N$ given by
$\PT(n)= \#  \left ( \PalT(\uu) \cap \L_n(\uu) \right )$.
If $a \in \A$, $w \in \PalT(\uu)$,
 and $aw\Theta(a) \in \Lu$, then $aw\Theta(a)$ is said to be a~\textit{$\Theta$-palindromic extension} of $w$ in $\uu$.
The set of all $\Theta$-palindromic extensions of $w$ is denoted by $\PextT(w)$.

\section{Generalized Thue-Morse words and dihedral groups}

In this section we show that $\L(\tt_{b,m})$ is closed under all elements of an explicit group $G \subset AM(\A^*)$.
Fix $b \geq 2$ and $m \geq 1$.
In what follows, to ease the notation, we denote $\varphi = \varphi_{b,m}$,
 the alphabet is considered to be $\Z_m$,
and letters are expressed modulo $m$.

For all $x \in \Z_m$ denote by $\Psi_{x}$ the antimorphism given by
$$
\Psi_{x}(k) := x - k \quad \text{for all } k \in \Z_m
$$
and by $\Pi_{x}$ the morphism given by
$$
\Pi_{x}(k) := x + k \quad \text{for all } k \in \Z_m.
$$
Denote by $D_m$ the set $D_m = \{ \Psi_{x} \bigm | x \in \Z_m \} \cup \{ \Pi_{x} \bigm | x \in \Z_m \}$.
It is easy to show that $D_m$ is a group and can be generated by $2$ elements, for instance one can choose $\Pi_{1}$ and $\Psi_{0}$.
Since the order of $\Pi_1$ is $m$, $\Psi_0$ is an involution, and $\Psi_0 \Pi_1$ is also an involution,
 $D_m$ is isomorphic to the dihedral group of order $2m$.

The following property of $\varphi$ will help us to prove the next proposition.

\begin{enumerate}[Property I]
 \item \label{p_komutovani} For all $x \in \Z_m$, we have $\Pi_{x} \varphi = \varphi \Pi_{x}$ and $\Psi_{x} \varphi = \varphi \Psi_{x+b-1}$.
 \begin{proof}
 It follows directly from the definitions of $\varphi$, $\Pi_x$ and $\Psi_x$.
 \end{proof}
\end{enumerate}

\begin{prop} \label{TM_closed_under_Dm}
The language of $\tt_{b,m}$ is closed under all elements of $D_m$.
\end{prop}

\begin{proof}
We show the claim by induction on the length $n$ of factors.
For $n=1$ the statement is easy to verify since $\L_1(\tt_{b,m}) = \Z_m$.
Suppose now the claim holds for factors of length $n$ and take $w \in \L_{n+1}(\tt_{b,m})$.
It is clear that there exists a factor $v$, $1 \leq |v| \leq n$, such that $w$ is a factor of $\varphi(v)$.
Let $x \in \Z_m$, then using Property \ref{p_komutovani}, one has
$ \Pi_{x} \varphi (v) = \varphi \Pi_{x} (v)$.
Since we supposed $\Pi_{x} (v) \in \L(\tt_{b,m})$, it is clear that $\Pi_{x}(w)$ is a factor of $\tt_{b,m}$
Using again Property \ref{p_komutovani} for $\Psi_x$, we have also
$\Psi_{x} \varphi (v) = \varphi \Psi_{x+b-1} (v)$, and thus $\Psi_{x} (w)$ is a factor of $\tt_{b,m}$.
\end{proof}

\section{$D_m$-richness of $\tt_{b,m}$}

In this section we show that the word $\tt_{b,m}$ is $D_m$-rich.
Let $\pi: \Z_m \mapsto \Z_m$ denote the permutation defined for all $k \in \Z_m$ by
$$
\pi(k) = \text{the last letter of } \varphi(k) = k + b - 1 = \Pi_{b-1}(k).
$$
Denote by $q$ the order of $\pi$, i.e., the smallest positive integer such that $q(b-1) \equiv 0 \mod m$.
We say that a factor $v = v_0v_1 \ldots v_{s-1} \in \L(\tt_{b,m})$ is an \textit{ancestor}
of a factor $w \in \L(\tt_{b,m})$ if $w$ is a factor of $\varphi(v)$ and is not a factor of $\varphi(v_1 \ldots v_{s-1})$
or $\varphi(v_0 \ldots v_{s-2})$.

The following properties of $\varphi$ and $\tt_{b,m}$ can be easily deduced.

\begin{enumerate}[Property I]
  \setcounter{enumi}{1}
  \item \label{p_uniform} $\varphi$ is uniform, i.e., for all $k,\ell \in \Z_m$, $|\varphi(k)| = |\varphi(\ell)| = b$.
  \item \label{p_bifix} $\varphi$ is bifix-free, i.e., for all $k,\ell \in \Z_m$, $k \neq \ell$, the first  letter of $\varphi(k)$ differs from the first letter of  $\varphi(\ell)$, and the same holds for the last letter.
  \item \label{p_L2} $\L_2(\tt_{b, m}) = \{ \pi^k(r-1)r \mid r \in \Z_m, 0 \leq k \leq q-1 \}$.
\begin{proof}
Denote by $L_0 := \{ (r-1)r \mid r \in \Z_m \}$.
It is clear that $L_0 \subset \L_2(\tt_{b, m})$.
For all $i$, denote by $L_{i+1}$ the set of factors of length $2$ of the words $\varphi(w)$, $w \in L_i$.
From the definition of $\pi$, it is clear that $L_i = \{ \pi^k(r-1)r \mid r \in \Z_m, 0 \leq k \leq i \}$.
The definition of $q$ then guarantees $L_{q-1} = L_q$ and the equality $L_{q-1} = \L_2(\tt_{b, m})$ follows from the construction of the sets $L_i$.
\end{proof}
  \item \label{p_L3} $\L_3(\tt_{b,m}) = \{ \pi^k(t-1) t(t+1) \mid t \in \Z_m, 0 \leq k \leq q-1 \} \cup \{ (t-1)t \pi^{-k}(t+1) \mid t \in \Z_m, 0 \leq k \leq q-1 \}$.
\begin{proof}
It follows from Property \ref{p_L2} and the definition of $\varphi$.
\end{proof}
  \item \label{p_pali2} For all words $w$ of length $1$ or $2$, there exists exactly one $x$ such that $w$ is a $\Psi_x$-palindrome.
\begin{proof}
It follows directly from the definition of $\Psi_x$.
\end{proof}
  \item \label{p_prenos_psi} If for $x \in \Z_m$, the word $w \in \L(\tt_{b,m})$ is a $\Psi_x$-palindrome, then the factor $\varphi(w)$ is a $\Psi_{x-b+1}$-palindrome.
\begin{proof}
It follows directly from Property \ref{p_komutovani}.
\end{proof}
  \item \label{p_prenos_nu} Let $w \in \L(\tt_{b,m})$ be a BS factor and $\nu \in D_m$.
Then $\nu(w)$ is BS factor and $\b(w) = \b(\nu(w))$.

Moreover, if $w$ is a \mbox{$\Theta$-palindrome} for some antimorphism $\Theta \in D_m$, then $\nu(w)$ is a $\Theta'$-palindrome for some $\Theta' \in D_m$.
\begin{proof}
Property \ref{p_bifix} and \Cref{TM_closed_under_Dm} guarantee the first part of the statement.
The second part of the statement can be verified by setting $\Theta' = \nu\Theta\nu^{-1}$.
\end{proof}
  \item If $w = w_0 \ldots w_{s-1} \in \L(\tt_{b,m})$ and there is an index $i$ such that $w_{i+1} \neq w_i + 1$, then $w$ has exactly one ancestor.
\begin{proof}
It follows directly from the definition of $\varphi$ and Property \ref{p_uniform}.
\end{proof}
  \item \label{p_2bplus1} Let $b \not \equiv 1 \mod m$
and $w \in \L(\tt_{b,m})$. If $|w| > 2b$, then $w$ has exactly one ancestor.
\begin{proof}
Take $|w| = 2b+1$.
Suppose that there is no index $i$ such that $w_{i+1} \neq w_i + 1$, i.e., $w = k(k+1)\ldots (k+2b)$ for some integer $k$.
Since every ancestor of $w$ is of length $3$,
it implies that there exists a factor $v \in \L_3(\tt_{b,m})$ such that
$v = \ell (\ell + b) (\ell + 2b)$ for some $\ell$.
Since $b \not \equiv 1 \mod m$,
it is a contradiction with Property \ref{p_L3}.
\end{proof}
  \item \label{p_prefix} If $w \in \L(\tt_{b,m})$ is BS and $|w| \geq b$,
  then there exist letters $x$ and $y$ such that $\varphi(x)$ is a prefix of $w$ and $\varphi(y)$ is a suffix of $w$.
\begin{proof}
The claim is a direct consequence of Property \ref{p_bifix} and the definition of $\varphi$.
\end{proof}
\end{enumerate}

These properties are used to prove the next two lemmas.
The first lemma summarizes the bilateral orders and $\Theta$-palindromic extensions of longer BS factors.

\begin{lem} \label{TM_dlouhe_BS}
Let $b \not \equiv 1 \mod m$.
Let $w \in \L(\tt_{b,m})$ be a BS factor such that $|w| \geq 2b$.
Then there exists a BS factor $v$ such that $\varphi(v) = w$.
Furthermore, $\b(w) = \b(v)$.

If $v$ is a $\Ta$-palindrome, $\Ta \in D_m$, then there exists a unique $\Tb \in D_m$ such that
 $w$ is a $\Tb$-palindrome.
Moreover, $\# \Pext_{\Tb} (w) = \# \Pext_{\Ta} (v)$.
\end{lem}

\begin{proof}
Let $w$ be a BS factor of length $|w| \geq 2b$.

If $|w| > 2b$, the existence of a unique ancestor $v$ follows from Property \ref{p_2bplus1}.
If $|w| = 2b$ and if there exists an ancestor $v \in \L_3(\tt_{b,m})$,
we have a contradiction with Properties \ref{p_prefix} and \ref{p_L3}.
Thus, if $|w| = 2b$, then there exists a unique ancestor $v \in \L_2(\tt_{b,m})$ such that $\varphi(v) = w$.

Property \ref{p_bifix} guarantees $v$ is BS and $\b(w) = \b(v)$.

Suppose $v$ is a $\Psi_x$-palindrome for some $x \in \Z_m$.
According to Property \ref{p_prenos_psi}, $w$ is a \mbox{$\Psi_{x-b+1}$-palindrome}.
The fact that there is no other such antimorphism follows from Property \ref{p_pali2}.

The equality $\# \Pext_{\Psi_x} (v) = \# \Pext_{\Psi_{x-b+1}} (w)$ follows again from Property \ref{p_prenos_psi}.
\end{proof}

Thanks to the last lemma, we have to evaluate only the bilateral orders
and number of palindromic extensions of shorter factors.
The next lemma exhibits these values for concerned lengths of BS factors.

\begin{lem} \label{TM_kratke_BS}
Let $b \not \equiv 1 \mod m$.
Let $w$ be a BS factor of $\tt_{b,m}$ such that $1 \leq |w| < 2b$.
Let $\Theta \in D_m$ be the unique antimorphism such that $w = \Theta(w)$.
Then the values $\b(w)$ and $\# \PextT(w)$ are as shown in \Cref{TM_vycet_kratkych_BS}.
\end{lem}

\renewcommand{\tabcolsep}{0.2cm}
\renewcommand{\arraystretch}{1.3}

\begin{table}[h]
\begin{center}
\begin{tabular}{c|c|c}
$|w|$ & $\b(w)$ & $\# \PextT(w)$ \\ \hline \hline
$1 \leq |w| \leq b - 1$ & $0$ & $1$ \\ \hline
$|w| = b$ & $1$ & $2$ \\ \hline
$b+1 \leq |w| \leq 2b - 2$ & $0$ & $1$ \\ \hline
$|w| = 2b-1$ & $-1$ & $0$ \\
\end{tabular}
\end{center}
\caption{The bilateral order and number of palindromic extensions of a BS factor $w$ of $\tt_{b,m}$,  $b \not \equiv 1 \mod m$,  according to its length $|w|$. $\Theta \in D_m$ is the antimorphism such that $\Theta(w) = w$.}
\label{TM_vycet_kratkych_BS}
\end{table}

\begin{proof}
Let $w$, $0 < |w| < 2b$, be a BS factor and let $s$ denote its length.
It follows from Property \ref{p_prefix} that there exists $k \in \Z_m$ such that
$$
	w = k(k+1) \ldots (k+s-1).
$$
Since $w = \Pi_k \left( 01 \ldots (s-1)\right )$, thanks to Property \ref{p_prenos_nu}
we may take $w = 01 \ldots (s-1)$.

Let $\Theta \in D_m$ be the unique antimorphism such that $w = \Theta(w)$.
From the form of $w$ it follows that $\Theta = \Psi_{s-1}$.

We discuss the following cases distinguished by $s$.

\begin{enumerate}[a)]
  \item $s = 2b-1$. \newline
One can see that $w$ has exactly $2$ ancestors: the words $0b$ and $(m-1)(b-1)$.
(Both $0b$ and $(m-1)(b-1)$ belong to $\L_2(\tt_{b,m})$ since $\pi^{q-1}(b-1)b = 0b$ and $\pi^{q-1}(b-2)(b-1) = (m-1)(b-1)$.)
The only pairs of letters $x$ and $y$ such that $xwy \in \L(\tt_{b,m})$ are $m-1+b-1$ and $2b-1$, and $m-1$ and $b$.
Therefore, $\b(w) = 2 - 2 - 2 + 1 = - 1$.
Since $\Theta = \Psi_{2b-2}$, one can see that no extension $xwy$ is a~$\Theta$-palindrome, i.e., $\# \PextT(w) = 0$.
 \item $b+1 \leq s \leq 2b-2$. \newline
One can deduce that $w$ has $2b - s + 1$ ancestors, namely
$( s -2b + i)(s - b + i)$ for $0 \leq i < 2b - s$.
(Again, all these words are factors of $\tt_{b,m}$ since $\pi^{q-1}(i-1) = i-b$ for all $i$.)
The only extensions $xwy$ appearing in $\L(\tt_{b,m})$ are $(m-1)w(s-b+1)$, $(m-1)ws$, and $(m+b-2)ws$.
Thus, we have $\b(w) = 3 - 2 - 2 + 1 = 0$ and $\# \PextT(w) = 1$ ($(m-1)ws$ is a $\Theta$-palindrome).
\item $s = b$. \newline
The ancestors of $w$ are the factors $0$ and $(i-b)i$, for $0 < i < b$.
The extensions $xwy$ are $(m-1)wb$, $\pi^{\ell}(-1)w1$, and $(m+b-2)w\pi^{-\ell}(1)$,
where $0 \leq \ell < q$.
Therefore, $\b(w) = 2q - q - q + 1 = 1$.
Since the only $\Theta$-palindromes are $(m-1)wb$ and $(m+b-2)w1$, we have $\# \PextT(w) = 2$.
\item $2 \leq s \leq b-1$. \newline
The ancestors are the factors $m - 1, \ldots, m + s - b$
, the factor $0$, and the factors $(i-b)i$, for $0 < i < s$.
The extensions $xwy$ are $\pi^{\ell}(-1)ws$ and $(-1)w\pi^{-\ell}(s)$,
where $0 \leq \ell < q$.
Thus, $\b(w) = (2q - 1) - q - q + 1 = 0$ and since $(m-1)ws$ is the only $\Theta$-palindromic extension, we have $\# \PextT(w) = 1$.
\item $s = 1$. \newline
One can see that the extensions are $\pi^{\ell}(-1)w1$ and $(-1)w\pi^{-\ell}(1)$,
where $0 \leq \ell < q$.
Thus, $\b(w) = 0$ and $\# \PextT(w) = 1$ as in the previous case.
\end{enumerate}

\end{proof}

\begin{coro} \label{TM_bispecialy_zaver}
Let $w$ be a non-empty BS factor of $\tt_{b,m}$ and $b \not \equiv 1 \mod m$.
Then
\begin{enumerate}
  \item there exists a unique antimorphism $\Theta \in D_m$ such that $\Theta(w) = w$;
  \item $\b(w) = \# \PextT(w) - 1$.
\end{enumerate}
\end{coro}

\begin{thm} \label{TM_je_Gr}
The word $\tt_{b,m}$ is $D_m$-rich.
\end{thm}

\begin{proof}
First, let $b \not \equiv 1 \mod m$.
We show that
\begin{equation} \label{TM_pf_1}
\Delta \C (n) + 2m \ \  =  \sum_{\substack{\Theta \in D_m \\ \Theta \text{ antimorphism}}} \Bigl( \PT(n) + \PT(n+1) \Bigr) \qquad \hbox{for all } n \geq 1.
\end{equation}
Note that $\# D_m = 2m$.

First, we show the relation \eqref{TM_pf_1} for $n = 1$.
It is clear that $\C(1) = m$ and $\C(2) = qm$.
Thus, the left-hand side equals $qm - m + 2m = qm + m$.
According to Properties \ref{p_L2} and \ref{p_pali2}, it is clear that
$\displaystyle \sum_{\substack{\Theta \in D_m \\ \Theta \text{ antimorphism}}} \PT(1) = m$
and
$\displaystyle \sum_{\substack{\Theta \in D_m \\ \Theta \text{ antimorphism}}} \PT(2) = qm$.
Therefore, the right-hand side equals $qm + m$.

To show the relation \eqref{TM_pf_1},
we are going to verify for all $n \geq 1$ that the difference of the left-hand sides for indices $n+1$ and $n$
equals the difference of the right-hand sides for the same indices.
In other words, we are going to show that
\begin{equation} \label{TM_pf_2}
\Delta \C (n+1) - \Delta \C (n) = \sum_{\substack{\Theta \in D_m \\ \Theta \text{ antimorphism}}} \Bigl( \PT(n+2) - \PT(n) \Bigr)
\end{equation}
for all $n$.

According to the equation \eqref{C_je_suma_bw}, the left-hand side side can be written as
$$
\Delta^2 \C(n) = \sum_{\substack{ w \in \L_n(\tt_{b,m}) \\ w \text{ BS}}} \b(w)
$$
and for the right-hand side we can use
$$
\PT(n+2) - \PT(n) =  \sum_{\substack{ w \in \L_n(\tt_{b,m}) \\ w = \Theta(w)}} \bigl( \# \PextT(w) - 1 \bigr ).
$$
Using the fact that a non-bispecial $\Theta$-palindrome has exactly one $\Theta$-palindromic extension and \Cref{TM_bispecialy_zaver},
the relation \eqref{TM_pf_2} holds.

If $\tt_{b,m}$ is periodic, i.e., $b \equiv 1 \mod m$, the proof of the relation \eqref{TM_pf_2} can be done in a very similar way and  is left to the reader.

Finally, according to \Cref{G_rich_criterion} and Property \ref{p_pali2}, $\tt_{b,m}$ is $D_m$-rich.

\end{proof}

\section{Factor complexity}

To our knowledge, the factor complexity of the Thue-Morse sequence $\tt_{2,2}$ was described in
1989 independently in \cite{Br89} and \cite{LuVa} and of $\tt_{2,m}$ in \cite{TrSh}.

In the aperiodic case, to calculate the factor complexity, one can use \eqref{C_je_suma_bw} and \Cref{TM_kratke_BS,TM_dlouhe_BS}.
\Cref{tab:komplexita_tbm} shows the result: $\Delta \C(n)$ and $\C(n)$.
In the periodic case, the factor complexity is trivial: $\C(n) = m$ for all $n > 0$.

\renewcommand{\tabcolsep}{0.2cm}
\renewcommand{\arraystretch}{1.3}


\begin{table}[h!]
\begin{center}
\begin{tabular}{c|c|c}
$n$ & $\Delta \C(n)$ & $\C (n)$ \\ \hline
$0$ & $m -1$ & $1$ \\ \hline
$1$ & $qm - m$ & $m$ \\ \hline
$2 \leq n \leq b$ & $qm - m$ & $qm (n-1) - m (n-2)$ \\ \hline
$\begin{array}{c} b^k + 1 + \ell \\ k \geq 1, \, 0 \leq \ell < b^k - b^{k-1}\end{array}$ & $qm$ & $qm (n-1) - m(b^k - b^{k-1})$ \\ \hline
$\begin{array}{c} (2b-1)b^{k-1} + 1 + \ell \\ k \geq 1, \, 0 \leq \ell < b^{k+1} - 2b^{k} + b^{k-1} \end{array}$ & $qm - m$ & $qm (n-1) - m(b^k - b^{k-1} + \ell)$ \\ 
\end{tabular}
\end{center}
\caption{Values of $\Delta \C(n)$ and $\C(n)$ of the generalized Thue-Morse word $\tt_{b,m}$ for the aperiodic case $b \not \equiv 1 \mod m$.}
\label{tab:komplexita_tbm}
\end{table}

\section*{Acknowledgments}

This work was supported by the Czech Science Foundation
grant GA\v CR 201/09/0584, by the grants MSM6840770039 and LC06002
of the Ministry of Education, Youth, and Sports of the Czech
Republic, and by the grant of the Grant Agency of the Czech
Technical University in Prague grant No. SGS11/162/OHK4/3T/14.

\end{document}